\documentclass[reqno]{amsart}
\usepackage{a4}
\usepackage{amssymb}
\usepackage{mathrsfs}

\numberwithin{equation}{section}

\newtheorem{Th}{Theorem}[section]
\newtheorem{Cor}[Th]{Corollary}
\newtheorem{Lem}[Th]{Lemma}
\newtheorem{Prop}[Th]{Proposition}
\newtheorem{Claim}[Th]{Claim}

\newtheorem{claim-num}{Claim}

\def\av#1{\overline{#1}}

\def\a{\alpha}
\def\eps{\varepsilon}
\def\vk{\varkappa}
\def\l{\lambda}
\def\f{\varphi}
\def\s{\sigma}

\def\aut#1{\mathrm{Aut}(#1)}
\def\out#1{\mathrm{Out}(#1)}

\def\str#1{\langle#1\rangle}

\def\inn#1{\mathrm{Inn}(#1)}
\def\iat{\mathrm{IA}}

\def\inv{^{-1}}
\def\sle{\subseteq}

\def\Z{\mathbf Z}
\def\id{\mathrm{id}}

\renewcommand{\le}{\leqslant}
\renewcommand{\ge}{\geqslant}

\def\To{\Rightarrow}

\def\rank{\operatorname{rank} }

\def\vk{\varkappa}

\def\cB{\mathscr B}
\def\cC{\mathscr C}
\def\cD{\mathscr D}
\def\cE{\mathscr E}

\begin{document}

\title[Outer automorphisms of free groups]
{On the outer automorphism groups of free groups}
\author{Vladimir Tolstykh}
\address{
Department of Mathematics and Computer Science \\
Istanbul Arel University \\
34537 Tepekent - B\"uy\"uk\c{c}ekmece \\
Istanbul \\
Turkey}
\email{vladimirtolstykh@arel.edu.tr}
\maketitle

\begin{abstract}
We prove that the outer automorphism group
of a free group of countably infinite rank
is complete.
\end{abstract}

\section*{Introduction}

Khramtsov \cite{Khr} proved in 1990 that the outer automorphism
group $\out{F_n}$ of a free group $F_n$ of finite rank
$n \ge 3$ is complete; his proof used the results
on the structure of finite subgroups of the group $\out{F_n}.$ Later, in 2000,
Bridson and Vogtmann \cite{BrdVo} obtained another proof of completeness
of the group $\out{F_n}$ based on the results on the
action of the group $\out{F_n}$ on the Culler--Vogtmann
Outer Space.

The quoted results on the outer automorphism groups $\out{F_n}$
have been the natural development of the earlier results
by Dyer and Formanek \cite{DFo,DFo2,DFo3} on the automorphism groups of relatively
free groups of finite rank. One of the main results
obtained by Dyer and Formanek \cite{DFo} stated that the
automorphism group $\aut{F_n}$ is complete for
all $n \ge 2.$ In 2000, the author extended
this result to arbitrary free nonabelian groups
by showing that the automorphism group of any
infinitely generated free group is also complete \cite{To:Towers}.

The main result of the present paper states
completeness of the outer automorphism
group of a free group of countably infinite
rank. The similar result on free groups
of arbitrary infinite rank
seems to be also true, but
the main tool we rely upon in our paper---the small index property for relatively free algebras---is established only for free groups
of countably infinite rank. We say that
a relatively free algebra $\mathfrak F$ of infinite
rank $\vk$ has the {\it small index property},
if any subgroup $\Sigma$ of the automorphism
group $\Gamma=\aut{\mathfrak F}$ of index
at most $\vk$ contains the pointwise stabilizer
$\Gamma_{(U)}$ of a subset $U$ of $\mathfrak F$
of cardinality $< \vk$ \cite{To:Sm4Nilps}
(in the case when $\mathfrak F$ is countable,
this property coincides with the small index property
for {\it countable} structures, much studied
since the early 1980s.) As it has been demonstrated
in \cite{To:Sm4Nilps}, the small index property
is true for arbitrary infinitely generated free
nilpotent groups. The case of infinitely generated free groups
appears to be much harder, the validity of the
small index property being established
by Bryant and Evans in \cite{BrEv} only for free
groups of countably infinite rank.

Given a group $G$ and a relation $R \sle G^n$ on $G,$
we shall call $R$ {\it definable} in $G$ if it
can be characterized in $G$ in terms of
group operation; any definable relation on $G$
is invariant under all automorphisms of $G.$
If a relation $R$ on $G$ admits
a group-theoretic characterization
in $G$ involving relations $R^*_1,\ldots,R_n^*$ on $G,$ we shall say
that $R$ is definable in $G$ with
{\it parameters} $R^*_1,\ldots,R_n^*.$ Accordingly,
if all parameters $R^*_1,\ldots,R_n^*$
are invariant under a certain
automorphism $\s$ of $G,$ then
any relation $R$ which is definable in $G$
with parameters $R^*_1,\ldots,R_n^*$ is also invariant under $\s.$

Let $F$ denote a free group of infinite rank.

The paper is organized as follows. The bulk of the
first two sections is devoted to definability
of so-called extremal involutions in the group $\out F.$

An involution $\f \in \aut G$ where $G$ is a relatively
free group is said to be {\it extremal} if there is a basis
$\cB$ of $G$ such that $\f$ inverts some element of $\cB$
and fixes all other elements of $\cB.$ Respectively,
an involution $f \in \out G$ is called {\it extremal} if it is induced
(under the natural homomorphism $\aut G \to \out G$)
by an extremal involution from $\aut G.$ In the first
section we prove definability of the family
of involutions from $\out F$ which induce
diagonalizable involutions in the automorphism
group $\aut A$ of the abelianization $A$ of $F.$

In the second section we first prove definability
of a certain family of $1$-involutions in the group $\out F$ (involutions
that induce extremal involutions in the group $\aut A),$
and then work to obtain an explicit group-theoretic
characterization of the extremal involutions
in the group $\out F.$ Note that the extremal
involution are also used heavily in the paper \cite{Khr} by
Khramtsov, but he does not provide an explicit
description of them.

Let $\cB$ be a basis of $F.$ In the third section,
based on definability of extremal involutions,
we prove that an arbitrary automorphism $\Delta \in \aut{\out F}$
can be followed by a suitable inner automorphism to ensure that the
resulting automorphism $\Delta'$ fixes pointwise
the images of all $\cB$-finitary automorphisms
of the group $F$ in the group $\out F.$ At the final step, assuming
that $F$ is of countably infinite rank, and using
the small index property for $F$ we prove that
the automorphism $\Delta'$ is in fact trivial,
completing the proof.

\section{Involutions}

Let $F$ denote an infinitely generated free group. In this section we shall prove
definability of some classes of involutions of the group $\out F$
that will be used in the next section for group-theoretic characterization
of the extremal involutions.

To denote elements of the group $\out F$ we shall use lowercase Latin
letters and to denote elements of the group $\aut F$ lowercase Greek letters.
The natural homomorphism $\aut F \to \out F$ will be denoted
by $\widehat{\phantom a};$ whenever $r=\widehat \rho$
where $r \in \out F$ and $\rho \in \aut F,$ we shall say
that $r$ is {\it induced} by $\rho.$ The abelianization $F/F'$ of
$F$ will be denoted by $A.$

According to Theorem 3 of \cite{DSc}, which describes
elements of prime order in the automorphism
groups of free groups, if $\f \in \aut F$ is an involution,
then there is a basis $\cB$ of $F$ called a {\it canonical
basis} for $\f$ and a partition
\begin{equation} \label{Part0CanBas}
\cB=U \sqcup X \sqcup \bigsqcup_{x \in X} Y_x \sqcup \{z,z' : z \in Z \}
\end{equation}
of $\cB$ such that the action of $\f$ on $\cB$ is as follows:
\begin{align} \label{eqCanForm}
&\quad \f u = u, \quad u \in U, \\
& \begin{cases}
\f z = z' & z \in Z, \\
\f z'= z, & \\
\end{cases}  \nonumber \\
&\begin{cases}
\f x =x \inv,    & x \in X,\\
\f y =x y x\inv, & y \in Y_x.
\end{cases} \nonumber
\end{align}
(note that some sets participating in the right-hand
side of \eqref{Part0CanBas} might be empty).
As in \cite{To:Towers}, we shall call
any set of the form $\{x\} \cup Y_x$ where $x \in X=X(\cB)$
a {\it block} of $\cB;$ the cardinal $|Y_x|+1$
shall be called the {\it size} of the block. The set
$U=U(\cB)$ will be called the {\it fixed part}
of $\cB.$

Again as in \cite{To:Towers}, we call an involution
$\f \in \aut F$ a {\it soft} involution if the `permutational'
part $\{z, z' : z \in Z(\cB)\}$ is empty for all canonical bases $\cB$ for $\f.$
An involution $f \in \out F$ will be called {\it soft} if there is a soft involution
in $\aut F$ which induces $f.$ Observe that an involution
$\f \in \aut F$ is soft if and only if the image
of $\f$ under the natural homomorphism $\aut F \to \aut{A/2A},$
determined by the natural homomorphism $F \to A/2A,$
is trivial.

It is proved in \cite[Prop. 2.2]{To:Towers}
that soft involutions $\f,\psi$ are conjugate
in $\aut F$ if and only for all canonical
bases $\cB,\cC$ for $\f,\psi,$ respectively,
the action of $\f$ on $\cB$ is isomorphic
to the action of $\psi$ on $\cC,$ or, in
other words, $\f,\psi$ are conjugate if and only
if they have the same canonical form.
Bearing this result in mind, it is quite convenient
to call a soft involution $\f \in \aut F$ a {\it $\gamma$-involution}
if any canonical basis for $\f$ contains exactly $\gamma$ blocks.
Now if
$$
\cB=U \sqcup X \sqcup \bigsqcup_{x \in X} Y_x
$$
is a canonical basis for $\f \in \aut F,$
the tuple
\begin{equation}
\str{ |U|;\str{|Y_x|+1 : x \in X} }
\end{equation}
(the size of the fixed part followed by the tuple
of sizes of the blocks of $\cB$) determines
$\f$ up to conjugacy. We shall call the tuple
(\theequation) the {\it type} of $\f.$

Following a long-standing tradition, we divide
involutions of the group $\out F$ into involutions
of the {first kind} and involutions of the
{second kind}. By definition, an involution
$f \in \out F$ is called an involution of the {\it first
kind} if its preimage in $\aut F$ contains
an involution, and an involution of the
{\it second kind}, otherwise.

Given an element $a \in F,$ we shall denote by $\tau_a$ the
inner automorphism (conjugation) of $F$ determined by $a$:
$$
\tau_a(w)=awa\inv \qquad (w \in F).
$$

According to Theorem 5.2 of \cite{Cu}, which describes
outer automorphisms of free groups of prime order, the preimage of an involution $p \in \out F$
of the second kind contains an automorphism $\pi$ of $F$ such that there is a basis
$$
\cB = \{u\} \sqcup \{z,z' : z \in Z\}
$$
such that $\pi$ takes $u$ to itself and
\begin{alignat}2  \label{SecKind}
&\s z  &&=z',\\
&\s z' &&= u z u\inv. \nonumber
\end{alignat}
for all $z \in Z.$ Thus the square of $\pi$ is the inner
automorphism $\tau_u$ determined by a primitive element $u.$
Clearly, $p$ does not vanish under the natural
homomorphism $\eps : \out F \to \aut{A/2A}$ induced
by the natural homomorphism $\aut F \to \aut{A/2A}.$
It follows that soft involutions of $\out F$ are exactly those
that vanish under $\eps.$

\begin{Lem} \label{ConjCrit}
Let $\f$ be a soft involution of $\aut F,$
let $\cB$ be a canonical basis for $\f$
and let $X=X(\cB)$ be the set of elements
of $\cB$ that are inverted by $\f.$

{\rm (i)}  Suppose that an involution $g \in \out F$ is a conjugate of $\widehat \f.$
Then $g=\widehat \psi$ where a soft involution
$\psi \in \aut F$ is either a conjugate of $\f,$ or a conjugate
of $\tau_x \f,$ where $x$ is an element of
$X;$

{\rm (ii)} an element $s$ in $\out F$ commutes with $\widehat \f$ if and only
if there is a $\s \in \aut F$ in the preimage of $s$
such that
$$
\text{either } \s \f =\f \s, \text{ or } \s \f =\tau_x \f \s
$$
for a suitable $x \in X;$

{\rm (iii)} if for all $x \in X,$ the involutions
$\f,\tau_x \f$ are not conjugate {\rm(}have different canonical forms{\rm)} in $\aut F,$ then for every
$s \in \out F$ which commutes with $\widehat \f,$
there is a preimage $\s \in \aut F$ of $s$ which commutes
with $\f.$
\end{Lem}

\begin{proof}
(i) As we have observed above, any conjugate $g$ of $\widehat \f$
is a soft involution, since it must vanish, as $\f$
does, under the natural homomorphism $\out F \to \aut{A/2A}.$
Thus $g=\widehat \psi$ where $\psi \in \aut F$ is a
soft involution.

We have that
$$
\sigma \psi \sigma\inv =\tau_v \f
$$
where $\sigma$ is in $\aut F$ and $\tau_v$ is
conjugation determined by an element $v$ of $F.$ Since
$\psi$ is an involution, $\f v=v\inv.$ By Lemma 2.3 of
\cite{To:Towers}, there is a $w \in F$
such that $v =\f(w)w\inv$ or $v =\f(w) x w\inv,$
where $x$ is a member of a canonical basis for
$\f$ with $\f x =x\inv.$ Therefore,
$$
\text{either }\tau_{\f(w)}\inv \sigma \psi \sigma\inv \tau_{\f(w)}=\f,
\text{  or  }
\tau_{\f(w)}\inv \sigma \psi \sigma\inv \tau_{\f(w)}=\tau_x \f,
$$
as claimed.

(ii). Similarly to (i).

(iii). By (ii).
\end{proof}

Any involution $\theta \in \aut F$ which inverts all elements of
some basis of $F$ will be called a {\it symmetry};
any involution in $\out F$ whose full preimage in $\aut F$
contains a symmetry will be called a {\it symmetry}, too.

\begin{Prop} \label{Def_o_SoftInvs}
The family of all soft involutions is definable in the group $\out F.$
\end{Prop}

\begin{proof} The group-theoretic characterization
of soft involutions in $\out F$ will involve
the symmetries, and so we shall start with providing a definable
subset of $\out F$ which contains all symmetries.
Given a group $G$ we call a subset of $G$ {\it anti-commutative}
if its elements are pairwise not commuting.

The proof of Proposition 3.1 of \cite{To:Towers} demonstrates that if the canonical
form of an involution $\f \in \aut F$ is neither
\begin{alignat}2 \label{eqBeads}
\f u &=u,  &\quad u &\in U,\\
\f x &=x\inv,  & x & \in X,  \nonumber \\
\f y &=x y x\inv,  & y & \in Y_x, \nonumber
\end{alignat}
where $X$ is of (infinite) cardinality $\rank F$, all the sets $Y_x$ ($x \in X$) have
the same finite cardinality $k,$ and $|U| < k+1,$ nor
\begin{alignat}2 \label{eqSnakes}
\f u &=u,  &\quad u &\in U,\\
\f x &=x\inv,      &\quad &  \nonumber \\
\f y &=x y x\inv,  & y & \in Y, \nonumber
\end{alignat}
where the cardinal $|U|$ is finite, then the conjugacy
class of $\f$ in $\out F$ contains a pair of distinct commuting
elements $\f_1,\f_2$ with $\widehat \f_1 \ne \widehat \f_2.$
Thus the conjugacy class of $\widehat \f$ in $\out F$
is not anti-commutative.

Further, if $\f$ has the canonical form \eqref{eqSnakes},
the class of the involution $\tau_x\inv \f$ has neither
canonical form \eqref{eqBeads}, nor the canonical form \eqref{eqSnakes},
and hence the class of $\widehat \f =\widehat{\tau_x\inv \f}$
is not anti-commutative.

Next, the conjugacy class of any involution $p \in \out F$
of the second kind is also not anti-commutative.
As it has been explained above, there exists a $\pi \in \aut F$
in the full preimage of $p$ and a basis $\cB=\{u\} \sqcup \{z,z' : z \in Z\}$
of $F$ such that
\begin{alignat*}2
&\pi u  && = u,\\
&\pi z  &&=z', \qquad z \in Z,\\
&\pi z' &&= u z u\inv.
\end{alignat*}
A conjugate of $p=\widehat\pi$ commuting with
$p$ in $\out F$ can be constructed
as follows. Suppose that $\pi_1$ is an
automorphism of $F$ which acts on the basis $\cB$
as follows:
\begin{alignat*}2
&\pi_1 u  && = u,\\
&\pi_1 z  &&=z'{}\inv,\qquad z \in Z,\\
&\pi_1 z' &&= u z\inv u\inv.
\end{alignat*}
Clearly, $\widehat \pi_1$ is a conjugate of $\widehat \pi$
in $\out F.$ The product of $\pi$ and $\pi_1$ is
a soft involution followed by conjugation $\tau_u.$
This implies that $p=\widehat \pi$ and $\widehat \pi_1$
are commuting in $\out F.$

On the other hand, the class of symmetries (whose
elements all have the canonical form \eqref{eqBeads})
is anti-commutative. Indeed, observe that the image of any symmetry
under the natural homomorphism $\out F \to \aut A$ is the
automorphism $-\id_A$ of $A,$ and hence a product of
two symmetries in $\out F$ is induced by an $\iat$-automorphism
of $F$ (recall if $G$ is a group an automorphism
$\a \in \aut G$ is called an {\it $\iat$-automorphism}
of $G$ if it induces the trivial automorphism
on the abelianization of $G$). Meanwhile, no involution
in $\out F,$ as we have seen above, is induced by an
IA-automorphism of $F.$

Thus

\begin{Lem} \label{antiC}
The union of all anti-commutative conjugacy classes
of involutions in $\out F$ consists of soft
involutions and contains the
class of symmetries.
\end{Lem}

\begin{Claim} \label{NiceProds02Symms}
Every $\iat$-automorphism $\alpha$ of $\aut F$ which acts on
a suitable basis $\cB$ of $F$ of the form
$$
\cB = U \sqcup X \sqcup \bigsqcup_{x \in X} Y_x
$$
so that
\begin{alignat*}2
\alpha u &= u,         &\quad u &\in U, \\
\alpha x &= x,         & x &\in X,\\
\alpha y &= xy x\inv, & y &\in Y_x.
\end{alignat*}
is a product of two symmetries.
\end{Claim}

\begin{proof}
One immediately checks that $\alpha$ is the product
of the following symmetries $\theta$ and $\theta'$:
\begin{alignat*}3
\theta u &=u\inv,                  & \theta' u &=u\inv,               & u &\in U,\\
\theta x &=x\inv,                  & \theta' x &=x\inv,               & x &\in X,\\
\theta y &=y\inv, \qquad           & \theta' y &=x\inv y\inv x,\quad & y &\in Y_x.
\end{alignat*}
\end{proof}

Theorem 3 of \cite{DSc} and Theorem 5.2 of \cite{Cu}, whose partial cases concerning
involutions we have quoted above, imply that any element of $\out F$ of
order three does not vanish under the natural homomorphism $\out F \to \aut{A/2A}.$
Thus, by Lemma \ref{antiC}, for every soft involution $f \in \out F,$
no product of the form $a_1 f_1 a_2 f_2$
where $f_1,f_2$ are conjugates of $f$ and each of $a_1,a_2$ is a product
of two involutions from an anti-commutative
conjugacy class, is of order three.

Conversely, if an involution $p \in \out F$ is not soft, there is always a product
of the form $a_1 p_1 a_2 p_2,$ where $p_k$ is a conjugate of $p$ and
$a_k$ is a product of two involutions from an anti-commutative conjugacy class $(k=1,2)$,
which is of order three.
Observe first that if $p=\widehat \pi$ is an involution of the second kind (see
descripiton of the action of $\pi$ on $F$ in \eqref{SecKind})
then, by Claim \ref{NiceProds02Symms}, there is a product $a$
of two symmetries such that $ap$ is an involution of the
first kind which is not soft.

Assume then that $p \in \out F$ is an involution of the first kind
which is not soft and that $p=\widehat \pi$ where $\pi$
is an involution in $\aut F.$ We claim that there is a product $p_1 p_2$
of conjugates of $p$ of order three. Consider a canonical
basis $\cB$ for $\pi.$ Choose an element $b_1$ in the `permutational'
part $Z(\cB)$ of $\cB.$ Let $b_2 =\pi(b_1)$ and let $\cC=\cB \setminus \{b_1,b_2\}.$
We define conjugates $\pi_1,\pi_2$ of $\pi$ as follows:
$$
\pi|_{\cC} =\pi_1|_\cC = \pi_2|_\cC
$$
and
$$
\begin{cases}
\pi_1(b_1)=b_2\inv,\\
\pi_2(b_2)=b_1\inv,
\end{cases}
\quad
\begin{cases}
\pi_2(b_1)=b_2b_1,\\
\pi_2(b_2)=b_2\inv.
\end{cases}
$$
It is easy to see that the product $\widehat \pi_1 \widehat \pi_2$
of two conjugates of $p$ is of order three in the group $\out F.$
\end{proof}

\section{Extremal involutions}

The goal of this section is to find a group-theoretic characterization
of the class of the extremal involutions in the group $\out F.$
We call a soft involution $\f \in \aut F$ {\it extremal} \cite{To:Towers}
if given any canonical basis $\cB$ for $\f,$
$\f$ inverts exactly one element of $\cB$ and fixes all other
elements of $\cB$. Respectively, an involution $f \in \out F$
is extremal if it is induced by an extremal involution in $\aut F.$
Clearly, any extremal involution is a $1$-involution.

\begin{Lem} \label{lemFilter4fSnakes}
Let $K \sle \out F$ be a conjugacy class of soft involutions
satisfying the following condition:
\begin{equation} \label{Filter4fSnakes}
\begin{minipage}{0.8\linewidth}
the product $K^2=K \cdot K$ contains
exactly one class of involutions and
no element of $K$ is a square in the
group $\out F.$
\end{minipage}
\end{equation}
Then $K$ is a class of $1$-involutions induced
by involutions from $\aut F$ with the canonical form
\begin{alignat*}2
\f u &= u,             & u &\in U, \\
\f x &=x \inv,         &   &\\
\f y &=x y x\inv,\quad & y &\in Y
\end{alignat*}
where the cardinal $|Y|$ is finite.
\end{Lem}

\begin{proof} We start with a simple observation which
we shall use a number of times in the proof.

\begin{Claim} \label{TrivCovs}
Suppose that $\psi \in \aut F$ is a $1$-involution
with a canonical basis $V \sqcup \{x\} \sqcup Y$ such that
\begin{alignat*}2
\psi x &=x\inv,           && \\
\psi y &= x yx\inv,\qquad && y \in Y,  \\
\psi v &= v,              && v \in V,
\end{alignat*}
where $|Y|+1 \le |V|$ and let $L$ be the conjugacy
class of $\widehat\psi$ in the group $\out F.$ Then

{\rm (i)} any soft involution in $\aut F$
of type $\str{\vk;|Y|+1,|Y|+1}$ is a product
of two conjugates of $\psi;$

{\rm (ii)} if $Y$ is infinite, the product $L^2=L \cdot L$ contains
infinitely many conjugacy classes of soft involutions.
Consequently, the class $L$ does not meet
the condition \eqref{Filter4fSnakes}.
\end{Claim}

\begin{proof} (i) Let $W$ be a subset of $V$ of cardinality
$|Y|+1$ and let $w \in W.$ Then the product $\psi \psi_1$
where a conjugate $\psi_1$ of $\psi$ is defined so that
\begin{alignat*}2
\psi_1 x &=x,     && \\
\psi_1 y &=y,             && y \in Y, \\
\psi_1 w &=w\inv  && \\
\psi_1 u &= w u w\inv,\quad && u \in W \setminus \{w\},  \\
\psi_1 v &= v,             && v \in V \setminus W
\end{alignat*}
is a soft involution of type $\str{\vk; |Y|+1,|Y|+1}$
(cf. also Claim 3.2 of \cite{To:Towers} which demonstrates
that there is a product of two conjugates of a soft involution
of type $\str{*;\l,\mu,\ldots}$ where $\l < \mu$
having type $\str{*;\l,\l,\ldots}$).

(ii) Take a natural number $n,$ take a subset $T$ of $Y$ of cardinality $n,$ write $Z$ for $Y \setminus T,$
and take an element $z \in Z.$ Then the involution $\psi_2 \in \aut F$ such that
\begin{alignat*}2
\psi_2 x &=x,     && \\
\psi_2 y &= y,             && y \in T, \\
\psi_2 z &=z\inv  && \\
\psi_2 y &= z y z\inv,\quad && y \in Z \setminus \{z\},  \\
\psi_2 v &= v,             && v \in V.
\end{alignat*}
is a conjugate of $\psi$ whose product with $\psi$ is of type
$$
\str{\vk; |T|+1, |Z|}=\str{\vk; n+1,|Y|}.
$$
Now given distinct natural numbers $m,n,$ the images
in $\out F$ of soft involutions of types $\str{\vk;m+1,|Y|},$
$\str{\vk;n+1,|Y|}$ of $\aut F$ are not conjugate in $\out F$
(Lemma \ref{ConjCrit}).
\end{proof}

Let $\f \in \aut F$ be a soft involution
and let $\cB$ be a canonical basis for $\f.$
Suppose that the conjugacy class $K$ of the involution
$\widehat \f$ satisfies the condition \eqref{Filter4fSnakes}
and suppose, towards a contradiction, that $\f$ is not a $1$-involution.

Let us assume first that

I. {The canonical basis $\cB$ contains three blocks of different
sizes}, say
$$
\cB_a=\{a\} \cup Y_a,\quad \cB_b=\{b\} \cup Y_b,\text{ and } \cB_c=\{c\} \cup Y_c
$$
of size $\l,\mu,\nu,$ respectively,
where $\l < \mu < \nu.$ Consider the involution $\f_1=\tau_c\inv \f$ which
also induces the involution $\widehat \f.$
Then the basis
\begin{equation}
\cB_1=U(\cB) \sqcup \bigsqcup_{x \in X(\cB), x \ne c} \{c^{-1}x : x \in X(\cB)\} \sqcup \{c\} \sqcup \bigsqcup_{x \in X(\cB)} Y_x
\end{equation}
is a canonical basis for $\f_1$ and the fixed part of $\cB_1$
is equal to the nonempty set $Y_c.$ Observe, for the future use,
that $\{c\} \cup U$ where $U=U(\cB)$ is a block of
the basis $\cB_1$ of size $|U|+1.$

Now as $|\cB_a| \le |Y_c|$ and $|\cB_b| \le |Y_c|,$
it is easy to rework an example in part (i) of Claim \ref{TrivCovs}
to show that there is a products of two conjugates of $\f_1$ having
type $\str{\vk;\l,\l}$ and there is a product of
two conjugates of $\f_1$ having type $\str{\vk;\mu,\mu}.$
Accordingly, the product $K^2=K \cdot K$ of the
conjugacy class $K$ of $\widehat \f$ in $\out F$
contains at least two conjugacy classes of (soft) involutions,
thereby failing to meet the condition \eqref{Filter4fSnakes}.

II. Suppose therefore that all blocks of the canonical basis $\cB$
are either of size $\l,$ or of size $\mu$ where $\l < \mu$
and let
$$
\cB_a=\{a\} \cup Y_a,\quad \cB_b=\{b\} \cup Y_b
$$
be blocks of $\cB$ of size $\l$ and $\mu,$ respectively. Regarding
to the fixed part $U$ of $\cB,$ we then conclude that
\begin{equation} \label{my_fixie}
|U|+1 = \l, \text{ or } |U|+1 = \mu.
\end{equation}
Indeed, otherwise the cardinals $|U|+1,\l,\mu$ are pairwise distinct,
and, arguing as above, we see that in the case when $\mu < |U|+1$ the product $K^2$ contains
images of soft involutions of types $\str{\vk;\l,\l}$ and $\str{\vk;\mu,\mu},$
and in the case when $|U+1| < \mu$ the said product contains
images of soft involutions of types $\str{\vk;\l,\l}$ and $\str{\vk;|U|+1,|U|+1}.$
Therefore there exist at least two conjugacy of involutions
in $K^2,$ a contradiction.

(a). Let $\mu < \vk.$ We then claim that the involution $\widehat \f$
is a square in the group $\out F,$ contradicting \eqref{Filter4fSnakes}.
It is certainly
true in the case when both the cardinality of the set
of $\l$-blocks and the cardinality of the
set of $\mu$-blocks are even (in particular,
infinite)
when $\f$ is a square in the group $\aut F$
(see an example in Remark 4.11 of \cite{To:Towers}).

Now let the cardinality of the
set of $\l$-blocks be odd. Then the
cardinality of the set of $\mu$-blocks
is equal to $\vk$ (notice also that due to the assumption $\mu < \vk$
both cardinalities cannot be odd). However,
\begin{itemize}
\item if $|U|+1=\l,$ the number of $\l$-blocks is incremented
by $1$ after multiplication of $\f$ by
$\tau_b\inv;$

\item if $|U|+1=\mu,$ the number of $\l$-blocks is decremented by $1$ after multiplication
of $\f$ by $\tau_a\inv,$
\end{itemize}
while the cardinality of the set of $\mu$-blocks
remains equal to $\vk$ in either case. Thus either $\tau_b\inv \f,$
or $\tau_a\inv \f$ is a square.

The case when the number of $\mu$-blocks is odd is similar.

(b). Let then $\mu=\vk$ and let the fixed part $U$ of $\cB$ be also of cardinality
$\vk$ (equivalently, $|U|+1=\vk$ as in \eqref{my_fixie}). But then an appropriate modification of the proof of
part (ii) of Claim \ref{TrivCovs} shows that the product $K^2$
contains infinitely many conjugacy classes of involutions,
and hence cannot satisfy the condition \eqref{Filter4fSnakes}.

(c). Let the basis $\cB$ have at least two $\vk$-blocks. Then
the involution $\tau_b\inv \f$
has the fixed part of cardinality $\vk$
and a block of size $\vk.$ The part is therefore
reduced to the previous part (b).

(d). Let $\cB_b$ be the only block of size $\vk$
and let $|U|+1=\l.$ Now any canonical
basis for the involution $\tau_b\inv \f$ has at least two $\l$-blocks
and the fixed part of cardinality $\vk.$ It is easy
to see that the product $K^2$ contains, say the images
of $4$-involutions of type $\str{\vk;\l,\l,\l,\l}$ (in addition
to the images of 2-involutions of type $\str{\vk;\l,\l}),$
again failing to meet the condition \eqref{Filter4fSnakes}.

III. Suppose that {all blocks of $\cB$ are of the same
size $\l.$} Notice that there are must be at least two
$\l$-blocks, since, as we assumed above, $\f$ is not
a $1$-involution.

In the case when $\l < \vk,$ we see that either
$\f$ is a square (if there are infinitely many
blocks), or the product $K^2$ contains
at least two conjugacy classes of soft involutions
(arguing as in part (d) above).

Finally, the case when $\l=\vk$ we argue as in part (c).

To complete the proof, we are going to prove that no $1$-involution in $\out F$ is a square. For the image of any $1$-involution $f$
under the natural homomorphism $\out F \to \aut A$ is an extremal involution
$\mathfrak f$ in the group $\aut A,$ one possessing a decomposition
$$
A = A^-_{\mathfrak f} \oplus A^+_{\mathfrak f}
$$
where $A^{\pm}_{\mathfrak f}$ are the eigengroups of $\mathfrak f$
associated with the eigenvalues $\pm 1.$ The eigengroup
$A^-_{\mathfrak f}$ is of rank 1 and is generated by an unimodular
element $a \in A$ with $\mathfrak f a=-a.$ Were $\mathfrak f=\mathfrak g^2$
where $\mathfrak g \in \aut A$ a square in the group
$\aut A,$ we would have that $\mathfrak g a=\pm a,$ since $\mathfrak g$ would
be commuting with $\mathfrak f,$ and hence preserving
the eigengroup $A^-_{\mathfrak f}.$  But then $\mathfrak g^2 a=a \ne \mathfrak f a,$
which completes the proof.
\end{proof}

The next statement provides useful
information on the conjugacy classes of $1$-involutions
(potentially) satisfying the condition \eqref{Filter4fSnakes}.

\begin{Lem} \label{Comm_with_1-invs}
Let $\f \in \aut F$ be a $1$-involution with the canonical
form
\begin{alignat*}2
\f x &=x\inv,           && \\
\f y &= x yx\inv,\qquad && y \in Y  \\
\f u &= u,              && u \in U
\end{alignat*}
where $|Y| < |U|=\vk.$ Then

{\rm (i)} an element $s \in \out F$ commutes with $\widehat \f$ if and only
if there is a $\s \in \aut F$ in the preimage
of $s$ which commutes with $\f;$

{\rm (ii)} if $\s \in \aut F$ commutes with $\f,$
then $\s x= v x^\eps v\inv$ where $v$ belongs to $\str U,$
the fixed-point subgroup of $\f,$ and $\eps=\pm 1;$

{\rm (iii)} if $|Y| \ge 1$ and if $\s \in \aut F$ is commuting
with $\f,$ then for every $y \in Y$ the element $\s y$
lies in a conjugate subgroup of $\str Y.$
\end{Lem}

\begin{proof}
(i) Clearly, involutions $\f$ and $\tau_x \f$ are not conjugate in $\aut F,$
since they have different canonical forms in $\aut F.$
Apply then part (iii) of Lemma \ref{ConjCrit}.

(ii) By Lemma 4.2 of \cite{To:Towers}.

(iii) First, notice that according to Lemma 2.3 (ii) of \cite{To:Towers}, whenever
$\f(z) = x zx\inv$ where $z \in F,$ then $z \in \str Y.$
It follows that if $\f(z) =x\inv z x,$ then $z \in x \str Y x\inv.$

By (ii), every automorphism of $F$ commuting
with $\f$ takes $x$ either to $v xv\inv,$ or to $v x\inv v\inv,$
where $v \in \str U.$ Suppose then that $\s(x)=v xv\inv.$ It follows that for every $y \in Y$
$$
\f( \s y) = \s(x) \s y \s(x)\inv = v x v\inv \s y v x\inv v\inv,
$$
or
$$
\f( v\inv \s y v) =  x (v\inv \s y v) x\inv,
$$
whence $v\inv \s y v \in \str Y$ and $\s y \in v \str Y v\inv.$

In the case when $\s x =v x\inv v\inv,$ arguing as above,
we obtain that $\s y \in vx \str Y x\inv v\inv$
for all $y \in Y.$
\end{proof}

\begin{Cor} \label{centerless}
The group $\out F$ is centerless.
\end{Cor}

\begin{proof}
Let $s=\widehat \s$ be in the center of $\out F.$
Since $s$ must commute with every extremal
involution $\widehat \f$ and since for every primitive
element $x \in F$ there exists an extremal
involution $\f \in \aut F$ with $\f x =x\inv,$
we get, by parts (i,ii) of Lemma \ref{Comm_with_1-invs},
that either
$$
\s x \sim x, \text{ or } \s x \sim x\inv,
$$
where $\sim$ is the conjugacy relation on $F,$ for every primitive element $x \in F.$
Now if $x,y$ is a couple of distinct elements of a basis of $F,$
we see, considering the action of $\s$ on the primitive element $xy,$
that the case when
$$
\s x \sim x \text{ and } \s y \sim y\inv
$$
is impossible. It follows that
\begin{itemize}
\item[(a)] either $\s x \sim x$ for all primitive elements of $F,$
\item[(b)] or $\s x \sim x\inv$ for all primitive elements
of $F.$
\end{itemize}

By the following statement, known as Grossman's lemma \cite[Lemma 1]{Gro},
if (a) holds, then $\s$ is an inner automorphism of $F,$ and hence
$s=1.$

\begin{Lem}[Grossman's lemma]  \label{Grossman}
If an automorphism $\alpha$ of a free group $G$ of rank at least two takes
every primitive element of a given free factor $U$ of $G$
with $\rank U \ge 2$ to a conjugate element,
then $\a$ acts on $U$ as an inner automorphism.
\end{Lem}

Suppose that (b) holds. Then the square $\s^2$ of $\s$
satisfies (a), and hence $\s^2$ is an inner automorphism of $F,$
which means that $s=\widehat \s$ is an involution in $\out F,$
namely, a symmetry. However, we know that distinct
symmetries are not commuting in $\out F.$
\end{proof}

\begin{Lem} \label{Comm_with_an_ext}
Let $\f \in \aut F$ be an extremal involution with
a canonical basis $\{x\} \sqcup U$ such that
\begin{alignat*} 2
\f x &=x\inv,\quad  && \\
\f u &=u,           && (u \in U).
\end{alignat*}
Then

{\rm (i)} if $s \in \out F$ commutes with $\widehat \f,$ then there
is a $\s \in \aut F,$ which induces $s,$ such that $\s x =x^{\pm 1}$
and $\s \str U = \str U;$

{\rm (ii)} the set $K^2$ where $K$ is the conjugacy class of $\widehat \f$
in $\out F$ contains exactly one conjugacy class of involutions;
accordingly, the class of extremal involutions satisfies
the condition \eqref{Filter4fSnakes};

{\rm (iii)} for every pair $g_1,g_2$ of conjugate $1$-involutions
in the centralizer $Z(\widehat \f)$ of $\widehat \f$ such that neither
is equal to $\widehat \f,$ $g_1$ and $g_2$ are conjugate in $Z(\widehat \f).$
\end{Lem}

\begin{proof}
(i) As $s$ commutes with $\widehat \f,$ there is a $\rho \in \aut F$
with $s=\widehat \rho$ which commutes with $\f$ (part (i) of Lemma \ref{Comm_with_1-invs}).
Then by part (ii) of Lemma \ref{Comm_with_1-invs}, $\rho x =v x^{\pm 1} v\inv.$
Also, $\rho$ must fix setwise the fixed-point subgroup of $\f,$ which means
that $\rho \str U=\str U.$ Now $\s =\tau_v\inv \rho$ induces
$s,$ takes $x$ to $x^{\pm 1}$ and fixes setwise the subgroup $\str U.$

(ii) Let $g \in \out F$ be an extremal involution commuting
with $\widehat \f.$ Consider an element $\psi \in \aut F,$ satisfying (i),
which induces $g.$

As $g$ is an involution, $\widehat \psi^2=1,$
or $\psi^2 = \tau_z$ for a suitable $z \in F.$ The automorphism
$\psi^2$ takes $x$ to itself, and then $z=x^k$ for a suitable
integer $k$ (since the centralizer
of a given primitive element of a free group consists
of powers of this element).
On the other hand, $\psi^2$ fixes setwise
the subgroup $\str U,$ and then $x^k=1.$ Hence $\psi^2 =\id,$
and $\psi$ is an extremal involution in $\aut F.$
In the case when $\psi x =x\inv,$ $\psi$ is equal to $\f.$
In the case when $\psi x =x,$ the product of $\f$ and $\psi$
is of type $(\vk; 1,1).$ It follows that the images of these
$2$-involutions are the only involutions in $K^2,$
where $K$ is the conjugacy class of $\widehat \f,$ as required.

Recall also that at the end of the proof of Lemma \ref{lemFilter4fSnakes}
we have demonstrated that no $1$-involution is
a square in the group $\out F.$ Thus the class
of extremal involutions in $\out F$ meets the condition
\eqref{Filter4fSnakes}, as claimed.

(iii). Let $g \in \out F$ be a $1$-involution which commutes
with $\widehat \f,$ but not equal to $\widehat \f.$ As in (ii), there is a $1$-involution
$\psi$ which induces $g$ and for which $\psi x=x^{\pm 1}$
and $\psi \str U=\str U.$ Clearly, the restriction
of $\psi$ on $\str U,$ an automorphism
of a free group, is also an involution. If this restriction
is the trivial automorphism of $\str U,$ then $\psi=\f,$ which
is impossible. Thus $\psi|_{\str U}$ is a $1$-involution
of $\aut{\str U}.$ It follows that $\psi x=x,$ for otherwise,
in the case when $\psi x=x\inv,$ canonical bases of $\psi$ would contain
at least two blocks,
which is again impossible. Therefore $\psi$ possesses a canonical
basis of the form $\{x\} \cup \cC,$ where $\cC$ is a basis
of $\str U,$ and the result follows.
\end{proof}

\begin{Prop} \label{Def0Exts}
The class of all extremal involutions
is definable in the group $\out F.$
\end{Prop}

\begin{proof} We have seen above (Lemma \ref{Comm_with_an_ext}) that the class
of extremal involution satisfies the condition
\eqref{Filter4fSnakes}. By Lemma \ref{lemFilter4fSnakes},
any other conjugacy class of soft involutions
which also meets the said conditions, must be
the class of a $1$-involution $\widehat \psi$ where $\psi$ is of type $(\vk;\l)$
with finite $\l \ge 2.$

However, unlike the class of extremal involutions, the class of $\widehat \psi$
does not satisfy the property
in part (iii) of Lemma \ref{Comm_with_an_ext}. For suppose that $\psi$
is an involution of the above type with the canonical form
\begin{alignat*}2
\psi x &=x\inv,           && \\
\psi y &= x yx\inv,\qquad && y \in Y  \\
\psi u &= u,              && u \in U
\end{alignat*}
where $Y$ is a nonempty finite set. Take an extremal involution
$\rho_1$ which inverts a certain element $y_0 \in Y$ and fixes
all other elements of $\cB = \{x\} \sqcup Y \sqcup U$ and an extremal
involution $\rho_2$ which inverts a certain element $u_0 \in U$
and fixes all other elements of $\cB.$ Then by part (iii)
of Lemma \ref{Comm_with_1-invs} no element $s \in Z(\widehat \psi)$ conjugates
$\widehat \rho_1$ and $\widehat \rho_2$ in $Z(\widehat \psi).$
\end{proof}

\section{Definability of finitary automorphisms}

Let us fix an extremal involution $f^* \in \out F.$ Suppose
that $f^*$ is induced by an extremal involution $\f^* \in \aut F$ which
inverts a primitive element $x \in F.$ Write $M$ for the
fixed-point subgroup of $\f^*.$ According to Lemma
\ref{Comm_with_an_ext} (i), any element in the centralizer $Z(f^*)$
of $f^*$ in $\out F$ is induced by an element in the centralizer
of $\f^*$ which takes $x$ either to $x,$ or to $x\inv$ and preserves
$M.$

It follows that all elements of the commutator
subgroup $Z(f^*)'$ of $Z(f^*)$ are induced by
elements of the subgroup $\Gamma_{(x),\{M\}} \cong \aut F$ of $\Gamma,$
that is, by those automorphisms of $F$ which fix $x$ and preserve $M$ setwise
(here and below we use the standard notation of the theory
of permutation groups: given a subset $U$ of $F,$ $\Gamma_{(U)}$ denotes the
pointwise stabilizer of $U$ in $\Gamma$ and $\Gamma_{\{U\}}$ the
setwise stabilizer of $U$ in $\Gamma$).

By Proposition 1.3 and Corollary 1.5 of \cite{To:Smallness} a nonidentity
automorphism $\s$ of $F$ is inner if and only if the cardinality
of the conjugacy class $\s^\Gamma$ of $\s$ in $\Gamma=\aut F$ is the least one
among the cardinalities of conjugacy classes of nonidentity elements
of $\aut F$:
$$
|\s^\Gamma| \le |\pi^\Gamma|
$$
for all $\pi \ne \id$ in $\Gamma=\aut F.$ Further, it is easy
to see that any inner automorphism of $F$ is in the commutator
subgroup of the group $\aut F.$ To see that take two distinct
primitive elements $z,t \in F$ that can be both included into
a basis of $F;$ then letting $\pi$ denote an automorphism
of $F$ which takes $z$ to $t,$ we see that the inner automorphism
$\tau_{zt\inv},$ determined by the primitive element $zt\inv,$
is equal to the commutator $\tau_z \pi \tau_z\inv \pi\inv.$

Accordingly, any nonidentity $h \in \out F$ such that
\begin{itemize}
\item[(a)] $h \in Z(f^*)';$
\item[(b)] $|h^{Z(f^*)}| \le |p^{Z(f^*)}|$ for all nonidenity
$p \in Z(f^*);$
\end{itemize}
is induced by an automorphism $\eta_t \in \aut F$ which fixes $x$ and
acts on $M$ as conjugation by a suitable element of $t \in M.$ We shall
call the automorphisms $\eta_t \in \aut F,$ as well as
their images in $\out F,$ {\it $(x,M)$-conjugations.} Clearly,
the group of all $(x,M)$-conjugations is isomorphic
to the group $\inn M \cong M,$ and hence to the group $F.$

Consider an automorphism $\Delta \in \aut{\out F}.$ Our aim is to
show that $\Delta$ can be followed by a number of inner automorphisms
of $\out F$ so that the resulting automorphism of $\out F$ is trivial,
provided that $F$ is of countable rank.

By Proposition \ref{Def0Exts}, $\Delta$ takes $f^*$ to another extremal involution in $\out F.$
As all extremal involutions in $\out F$ are conjugate, there
is an inner automorphism $T(r_1),$ determined by a suitable
$r_1 \in \out F,$ such that the automorphism
$$
\Delta_1=T(r_1) \Delta
$$
takes $f^*$ to itself.
Consequently, being definable with parameter
$f^*,$ the group $C$ of all $(x,M)$-conjugations
is invariant under $\Delta_1,$ and hence
$\Delta_1$ induces an automorphism of the group $C.$

Take a basis $B$ of the
free group $C.$ As the family of all bases
of $C$ is a definable object over $C,$
$\Delta_1$ takes $B$ to another basis of $C.$
But then a suitable inner automorphism $T(r_2),$ where $r_2 \in Z(f^*),$
takes the basis $\Delta_2(B)$ of $C$ back to $B.$
Thus
$$
\Delta_2 = T(r_2) \Delta_1
$$
takes $f^*$ to itself and stabilizes every element
of the group $C.$

In effect, $\Delta_2$ stabilizes each element
of the group $Z(f^*)',$ for it is invariant under $\Delta_2,$
and its elements are fully determined by their actions on $C.$

Consider a basis $\cB = \{x\} \cup \cC$ of $F$ where $\cC$
is a basis of $M.$ Recall that an automorphism $\s$ of $F$
is called {\it $\cB$-finitary} if $\s$ fixes all but
finitely many elements of $\cB.$ We shall call any automorphism of $F,$ which
preserves the basis $\cB$ as a set, a {\it $\cB$-permutational}
automorphism of $F.$

\begin{Lem} \label{Finitarki}
One of the automorphisms $\Delta_2,$ or $T(f^*) \Delta_2$
fixes pointwise the set of images of all $\cB$-finitary automorphisms of
$F$ in the group $\out F.$
\end{Lem}

\begin{proof}
Recall that the symmetric group
of any infinite set is perfect \cite{Ore}. Further, any finitary automorphism
of $F$ is in the commutator subgroup of $\aut F$ (say, since the automorphism
group of a free group of countably infinite rank is perfect \cite[Theorem C]{BrRom}).

It follows that  both $\Delta_2$ and
$T(f^*)\Delta_2$ fix the set of images in $\out F$ of all $\cB$-permutational automorphisms
in $\Gamma_{(x),\{M\}}$ (resp. of all $\cB$-finitary automorphisms
in $\Gamma_{(x),\{M\}}$) pointwise.

Thus to complete the proof it suffices to show that
either $\Delta_2,$ or $T(f^*)\Delta_2$ fixes
the set of images of all $\cB$-permutational
automorphisms of $F$ pointwise.

Take an element $y \in \cB \setminus \{x\}.$
The group of all $\cB$-permutational automorphisms
of $F$ is generated by all $\cB$-permutational automorphisms
which fix $x$ and by the $\cB$-permutational automorphism $\pi^*$ of $F$
which interchanges $x$ and $y$ and fixes pointwise
the set $\cB \setminus \{x,y\}$ (say, by the lemma on
page 580 of \cite{DiNeuTho}). So what we have to do
is to show that $\widehat{\pi^*}$ is stabilized either by $\Delta_2,$
or by $T(f^*)\Delta_2.$

Write $R$ for
\begin{quote}
the family of all elements of $Z(f^*)'$
which, under the conjugation action, fix $\widehat \eta_y$ and fix
the subgroup generated by $(x,M)$-conjugations $\widehat \eta_d$
determined by elements $d$ from $\cB \setminus \{x,y\}$
setwise.
\end{quote}
Take an element $r \in R.$ We claim that $r$ is induced
by an element of $\aut F$ which takes both
$x$ and $y$ to themselves and fixes the subgroup $\str{\cB \setminus \{x,y\}}$
setwise.

First, as $r \in Z(f^*)',$ $r$ is induced by an element
$\rho \in \Gamma_{(x),\{M\}}.$
Second, the condition $r \widehat \eta_y r\inv = \widehat \eta_y$
implies that
$$
\rho \eta_y \rho\inv = \tau_w \eta_y
$$
for some $w \in F.$ As $\rho(x)=\eta_y(x)=x,$
then $w=x^k$ where $k \in \Z.$
Taking an arbitrary
primitive element $t \in
M$ and forcing the maps
$\rho \eta_y \rho\inv$ and $\tau_{x^k} \eta_y$
act on $\rho(t),$ we obtain that
$$
\rho(y) \rho(t) \rho(y)\inv = \rho(y t y\inv)=x^k y \rho(t) y\inv x^{-k} \To
[y\inv x^{-k} \rho(y),\rho(t)]=1.
$$
It follows that
$\rho(y)=x^k y,$ which leads to $\rho(y)=y,$ because
$\rho$ has to preserve $M.$

Third, let $t \in \str{\cB \setminus \{x,y\}}.$ According
to the definition of $R,$ there are $t_1,\ldots,t_s \in \str{\cB \setminus \{x,y\}}$
such that
\begin{equation}
\rho \eta_t \rho\inv = \tau_v \eta_{t_1} \ldots  \eta_{t_s}
\end{equation}
for some $v \in F.$ Analyzing again the actions
of automorphisms involved in (\theequation) on $x,$ we see that
$v=x^k$ ($k \in \Z$), and keeping in mind that
each of the maps $\rho, \eta_t, \eta_{t_1},\ldots,\eta_{t_s}$
preserves $M,$ we obtain that $v=1.$ Now
$
\eta_{\rho(t)} = \eta_{t_1} \ldots  \eta_{t_s},
$
since $\rho$ preserves $M,$ which implies that $\rho(t)=t_1 \ldots t_s,$ whence the claim.

Let $S$ be the set of all $s \in \out F$ such that
\begin{itemize}
\item[(a)] $s$ commutes with each element
of $R;$
\item[(b)] if $r \in R$ is an extremal
involution then $s \in Z(r)'.$
\end{itemize}
As $S$ is definable with parameters $\{f^*, \widehat \eta_c : c \in \cB \setminus \{x\}\},$
$S$ is invariant under every automorphism of $\out F$
which fixes each of the parameters (in particular,
it is true both automorphisms $\Delta_2$
and $T(f^*)\Delta_2$).

We claim, on the other hand, that every element $s \in S$ is induced by an automorphism
of $F$ which preserves the subgroup $\str{x,y}$
generated by $x$ and $y$ setwise and fixes
the set $\cB \setminus \{x,y\}$ pointwise, and vice versa;
accordingly, $S \cong \aut{F_2}$ where $F_2$ is a 2-generator
free group. The sufficiency part is immediate due to
the fact all elements of $R$ are, as we have seen above, induced
by elements of $\Gamma_{(x,y), \{\str{\cB \setminus \{x,y\}}\}}.$

Conversely, take $s \in S$ and some $\s \in \aut F$ in
the preimage of $s.$

Consider a primitive element $t \in \str{\cB \setminus \{x,y\}},$
and let $\rho$ be an extremal involution which fixes
both $x$ and $y,$ preserves the subgroup $\str{\cB \setminus \{x,y\}}$
setwise, and inverts $t.$ Clearly, $r=\widehat \rho$
is in $R.$ Then, by (b), $\s$ takes $t$ to a conjugate of $t.$

But then by Grossman's lemma, quoted in Lemma \ref{Grossman}
above, the action of $\s$ on $\cB \setminus \{x,y\}$
coincides with the action of a certain inner
automorphism $\tau_u.$

Let $\widehat \rho$ be an element of $R$ where
$\rho \in \aut F$ fixes $x$ and $y,$ and preserves
the subgroup $\str{\cB \setminus \{x,y\}}$ setwise. By (a), there is a $z \in F$ with
\begin{equation}
\rho \s \rho\inv = \tau_z \s.
\end{equation}
Then for every $t \in \cB \setminus \{x,y\},$
$$
\tau_z \s(t)=\tau_z \tau_u(t)= z u t u\inv z\inv
$$
and
$$
\rho \s \rho\inv t=\rho \s(\rho\inv t)=\rho(u \rho\inv(t) u\inv)=\rho(u) t \rho(u)\inv,
$$
whence
$$
\rho(u) t \rho(u)\inv = z u t u\inv z\inv.
$$
It follows that the element $\rho(u\inv) z u$ commutes
with each $t$ in $\cB \setminus \{x,y\},$
whence
$$
\rho(u\inv) z u=1 \To z=\rho(u) u\inv.
$$

The equation (\theequation) implies then that
\begin{equation}
\rho (\tau_u\inv \s) \rho\inv =\tau_u\inv \s.
\end{equation}
The automorphism $\s_1=\tau_u\inv \s$ fixes $\cB \setminus \{x,y\}$
pointwise. Let $\s_1(x)=w(x,y;\av t)$ where $w$ is a reduced
word and $\av t$ is a tuple of elements of $\cB \setminus \{x,y\}.$ Suppose, towards
a contradiction, that $\av t$ is nonempty. Take then a $\cB$-permutational automorphism
$\rho,$ with $\widehat \rho \in R,$ which fixes both $x$ and $y$ and such that
$\pi(\av t) \cap \av t=\varnothing.$ By (\theequation),
$$
w(x,y;\rho \av t)=w(x,y;\av t),
$$
and hence $w$ has no occurrences of elements of $\cB \setminus \{x,y\},$
a contradiction. Similarly, $\s_1(y) \in \str{x,y}.$
Thus $s=\widehat \s_1=\widehat{\tau_u\inv \s}$ is induced
by an element of $\Gamma_{\{\str{x,y}\},(\cB \setminus \{x,y\})},$ as
claimed.

Thus $S \cong \aut{F_2}$ and $\Delta_2$ preserves $S$ setwise. As the group $\aut{F_2}$ is complete \cite{DFo},
the restriction of $\Delta_2$ on $S$ is an inner automorphism: there is
an $s_0 \in S$ such that
\begin{equation}
\Delta_2(s)=s_0 s s_0\inv
\end{equation}
for all $s \in S.$ Let $\s_0 \in \aut F$ with $s_0=\widehat \s_0$ be the automorphism
of $F$ which preserves $G=\str{x,y}$ setwise and preserves $\cB \setminus \{x,y\}$ pointwise.

The $(x,M)$-conjugation $\widehat{\eta_y}$ is in $S,$ and corresponds
to the inner automorphisms of $G$ determined by $y\inv.$
On the other hand, $\widehat{\eta_y}$ is stabilized
by $\Delta_2.$ In view of (\theequation), this means that $\s_0(y)=y.$
Further, $\Delta_2(f^*)=f^*$ implies that
$\s_0$ and $\f^*$ commute, whence $\s_0(x)=y^k x^{\pm 1} y^{-k}$ where
$k \in \Z.$ But $\Delta_2$ fixes also the extremal involution
$\widehat\f_y$ induced by the extremal involution $\f_y$ of $F$  which fixes pointwise
$\cB \setminus \{y\}$ and
inverts $y.$ Thus $\s_0(x),$ being an element
of the fixed-point subgroup of $\f_y,$ must not contain (as a reduced word) occurrences of $y.$
It follows that $k=0,$ and therefore $\s_0=\id,$ or $\s_0=\f^*.$

So $\Delta_2$ in the former case and $T(f^*)\Delta_2$
in the latter stabilizes each element of $S.$ Write then $\Delta_3$
for that one of the automorphisms $\Delta_2,$ or $T(f^*)\Delta_2$
which acts trivially on $S.$ Observe, for the future
use, that as $\Delta_2$ fixes $f^*$ and fixes pointwise
$Z(f^*)',$ so does $\Delta_3.$

Thus we see that $\Delta_3$ stabilizes $\widehat{\pi^*},$ where $\pi^*$ is the $\cB$-permutational
automorphism which interchanges $x$ and $y$ we have defined
above, and the proof is completed.
\end{proof}


From now on we shall assume that $F$ is of {\it countably
infinite} rank.

\begin{Th}
The outer automorphism group of a free group
of countably infinite rank is complete.
\end{Th}

\begin{proof}
We have seen above that the
group $\out F$ is centerless (Corollary \ref{centerless}).

 One of the key features of the group $F$
in the case when the rank of $F$ is countably infinite is
the small index property for $F$ \cite{BrEv}:
every subgroup $\Sigma$
of the group $\Gamma=\aut F$ of index $< 2^{\aleph_0}$
contains the pointwise stabilizer $\Gamma_{(U)}$ of a finite
subset of $F$ (note that the condition ``$< 2^{\aleph_0}$''
can be equivalently replaced with the condition ``$\le \aleph_0$'',
see \cite{Ho}).

\begin{Lem} \label{def0gammax}
Let $F$ be of countably infinite rank, let $t$
be a primitive element of $F,$ and let $\psi$
be an extremal involution which inverts
$t.$ Then the image $\widehat{ \Gamma_{(t)} }$
of the stabilizer $\Gamma_{(t)}$ of $t$ in $\Gamma=\aut F$
is definable in $\out F$ with parameter $\widehat\psi.$
\end{Lem}

\begin{proof} As $|\out F|=2^{\aleph_0},$
we have that
$$
|\out F : S| < 2^{\aleph_0} \iff  |\out F : S| < |\out F|
$$
for all subgroups $S$ of $\out F,$ whence the family of subgroups
of index $< 2^{\aleph_0}$ is a definable object
over $\out F.$

We claim that the subgroup $\widehat{ \Gamma_{(t)} }$
is the least subgroup of index $< 2^{\aleph_0}$ which contains
the commutator subgroup $Z(\widehat \psi)'$
of the centralizer $Z(\widehat \psi)$ of
$\widehat \psi.$

Indeed, let $S \ge Z(\widehat \psi)'$ be a subgroup
of $\out F$ of index $< 2^{\aleph_0}.$ Then the
full preimage $\Sigma$ of $S$ is also of index
$< 2^{\aleph_0}$ and contains $Z(\psi)'.$
Let $\cD$ be a basis of $F$ such that $t \in \cD$
and $\cD \setminus \{t\}$ is a basis of the
fixed-point subgroup of $\psi.$

We have seen above that $Z(\psi)'$ contains
the subgroup of all $\cD$-permutational
automorphisms which fix $t.$ Further,
due to the small index property for $F,$
the group $\Sigma$ contains the pointwise
stabilizer $\Gamma_{(\cE)}$
of a finite subset $\cE$ of $\cD;$
clearly, without loss of generality,
we may assume that $t \in \cE.$

We prove that $\Sigma \ge \Gamma_{(t)}.$
Take $\s \in \Gamma_{(t)}$ and let $\delta$
be a $\cD$-finitary automorphism with
\begin{equation}
\delta\inv \s \in \Gamma_{(\cE)} \le \Sigma;
\end{equation}
existence of $\delta$ is due to Lemma 1.4 of \cite{BrEv}.
As $\delta(t)=t$ and as $\delta$ is $\cD$-finitary,
there is a $\cD$-finitary automorphism $\delta_0 \in \Gamma_{(\cE)}$
and a $\cD$-permutational automorphism
$\pi$ which fix $t$ (a member of $\Sigma$) such that
$$
\delta=\pi \delta_0 \pi\inv.
$$
But then $\delta$ is in $\Sigma,$
and, by (\theequation), $\s \in \Sigma.$
Thus $\Sigma \ge \Gamma_{(t)},$
and hence
$
S=\widehat \Sigma \ge \widehat{\Gamma_{(t)}}.
$
\end{proof}

We shall continue to use the notation
we have introduced above. Take a primitive
element  $t \in F.$ It is easy to see that there exists an extremal
involution $\psi_t$ which inverts $t$
and which is a $\cB$-finitary automorphism
of $F.$ By Lemma \ref{Finitarki}, the automorphism
$\Delta_3$ of the group $\out F$ defined above
takes the involution $\widehat \psi_t$ to itself,
and by Lemma  \ref{def0gammax}, the subgroup
$\widehat{\Gamma_{(t)}}$ is invariant
under $\Delta_3.$

Now let $\widehat \s$ where $\s \in \aut F$
be an arbitrary element of $\out F.$
Given any primitive $t \in F,$ there is
a $\cB$-finitary automorphism $\delta_t$
such that $\delta_t\inv \s \in \Gamma_{(t)},$ whence
\begin{equation}
\widehat{\delta_t}\inv \cdot \widehat{\mathstrut \s} \in \widehat{ \Gamma_{(t)} }.
\end{equation}
By applying $\Delta_3$ to both parts of the relation (\theequation), we
see that
$$
\widehat{\delta_t}\inv \cdot \Delta_3(\widehat{\mathstrut \s}) \in \widehat{ \Gamma_{(t)} }.
$$
It follows that
$$
\Delta_3(\widehat{\mathstrut \s})\inv \cdot \widehat \s \in  \widehat{ \Gamma_{(t)} }
$$
for every primitive element $t.$ Let an automorphism $\rho \in \aut F$
induce $\Delta_3(\widehat{\mathstrut \s})\inv \cdot \widehat \s.$
Then
$$
\rho(t) \sim t
$$
for all primitive elements $t \in F.$ By applying Grossman's lemma for the last time, we see that
$\rho$ is an inner automorphism of $F,$ and that
$\Delta_3(\widehat \s)=\widehat \s.$

Then $\Delta_3,$ which has been obtained from
an arbitrary automorphism $\Delta$ of the
group $\out F$ after a series of multiplications
by inner automorphisms, is the identity
automorphism of the group $\out F.$
Consequently, all automorphisms
of the group $\out F$ are inner,
as claimed.
\end{proof}


\begin{thebibliography}{99}

\bibitem{BrdVo}
{M.~Bridson, K.~Vogtmann}.
{Automorphisms of automorphism groups of free groups},
{J. Algebra.}
{229}
(2000),
785--792.

\bibitem{BrEv}
{R.~M.~Bryant, D.~M.~Evans},
{The small index property for free groups and relatively free groups},
{J. London Math. Soc.}
55
(1997)
363--369.

\bibitem{BrRom}
{R.~M.~Bryant, V.~A.~Roman'kov},
{The automorphism groups of relatively free algebras},
{J.  Algebra}
209
(1998)
713--723.


\bibitem{Cu}
M.~Culler,
{Finite groups of outer automorphisms of a free group},
{\it in}
{Contributions to group theory},
Contemp. Math.
33
197-207
(1984).

\bibitem{DiNeuTho}
{J.~Dixon, P.~M.~Neumann, S.~Thomas},
{Subgroups of small index in infinite symmetric groups},
{Bull. London Math. Soc.}
18
(1986)
580--586.



\bibitem{DFo}
{J.~Dyer, E.~Formanek}.
{The automorphism group of a free group is complete.}
{J. London Math. Soc.}
{11}
(1975),
181--190.

\bibitem{DFo2}
{J.~Dyer, E.~Formanek}.
{Automorphism sequences of free nilpotent group of class two.}
{Math. Proc. Camb. Phil. Soc.}
{79}
(1976),
271--279.

\bibitem{DFo3}
{J.~Dyer, E.~Formanek}.
{Characteristic subgroups and complete automorphism groups.}
{Amer. J. Math.}
{99}
(1977)
713--753.

\bibitem{DSc}
J.~Dyer, G.~P.~Scott,
{Periodic automorphisms of free groups},
{Comm. Algebra}
3
(1975)
195--201.

\bibitem{Gro}
E.~K.~ Grossman,
{On the residual finiteness of certain mapping class groups},
{J. London Math. Soc},
9
(1974)
160--164.

\bibitem{Ho}
W.~Hodges,
Model Theory,
Cambridge University Press,
1993.

\bibitem{Khr}
D.~G.~Khramtsov,
{Completeness of groups of outer automorphisms of free groups},
{\it in}
{Group-theoretic investigations} (Russian),
Akad. Nauk SSSR Ural. Otdel.,
Sverdlovsk,
(1990)
128--143.


\bibitem{Ore}
O.~Ore,
{Some remarks on commutators},
{Proc. Amer. Math. Soc.}
{2}
(1951)
307--314.


\bibitem{To:Towers}
{V.~Tolstykh},
{The automorphism tower of a free group.}
{J. London Math. Soc. (2)}
{61}
(2000),
423--440.

\bibitem{To:Smallness}
{V.~Tolstykh},
{Small conjugacy classes in the automorphism groups of relatively free groups.}
{J. Pure Appl. Algebra}
{215}
(2011)
2086--2098.

\bibitem{To:Sm4Nilps}
{V.~Tolstykh},
{The small index property for free nilpotent groups},
{Comm. Alg.},
(to appear).
\end{thebibliography}
\end{document}